\documentclass[11pt]{amsart}
\usepackage{amsfonts}
\usepackage{amssymb}
\usepackage{amsmath}
\numberwithin{equation}{section}
\usepackage{dsfont}
\usepackage{color}
\usepackage{graphicx}
\newtheorem{theorem}{Theorem}[section]
\newtheorem{lemma}[theorem]{Lemma}
\newtheorem{proposition}[theorem]{Proposition}
\newtheorem{corollary}[theorem]{Corollary}

\newtheorem{conjecture}[theorem]{Conjecture}

\theoremstyle{definition}
\newtheorem{df}{Definition}

\newtheorem{remark}[df]{Remark}

\newcommand{\N}{\mathbb N}

\newcommand{\R}{\mathbb R}

\newcommand{\G}{\mathcal{G}}
\newcommand{\ha}{\symbol{94}}

\makeatletter
\@namedef{subjclassname@2020}{%
  \textup{2020} Mathematics Subject Classification}
\makeatother


\begin{document}
\author{Piotr Nowakowski}
\address{Faculty of Mathematics and Computer Science, University of Lodz,
ul. Banacha 22, 90-238 \L \'{o}d\'{z}, Poland
\\
ORCID: 0000-0002-3655-4991}
\email{piotr.nowakowski@wmii.uni.lodz.pl}
\title[Conditions for the difference set of a central Cantor set II]{Conditions
for the difference set of a central Cantor set to be a Cantorval. Part II.}
\date{}
\begin{abstract}
Let $C(a)\subset [0,1]$ be the central Cantor set
generated by a sequence $a =(a_{n})\in \left( 0,1
\right) ^{\mathbb{N}}$. It is known that the difference set $%
C(a)-C(a)$ of $C(a)$ can has one of three possible forms: a
finite union of closed intervals, a Cantor set, or a Cantorval. In the previous paper there was proved a sufficient condition for the sequence $a$ which implies that $C(a)-C(a)$ is a Cantorval. In this paper we give different conditions for a sequence $a$ which guarantee
the same assertion. We also prove a corollary, which provides infinitely many new examples of Cantorvals. 
\end{abstract}

\maketitle

\section{Introduction}

We call a set $C \subset \R$ a \emph{Cantor set} if it is compact, perfect and nowhere dense. 

First, let us introduce some basic notation. For $A,B\subset \mathbb{R}$ we denote by $A\pm B$ the set $\left\{ a\pm
b:a\in A,\,b\in B\right\} $. The set $A- B$ is called the \emph{algebraic difference} of sets $A$ and $B$, and the set $A-A$ is called the \emph{difference set} of a set $A$. We will also write $a+A$ instead of $\{a\} + A$ for $a \in \R$. The Lebesgue measure of a measurable set $A$ is denoted by $%
\left\vert A\right\vert $. If $I \subset \R$ is an interval, then we will denote by $l(I)$, $r(I)$, $c(I)$ the left endpoint, the right endpoint and the centre of $I$, respectively. 

By $i^{(n)}$ we denote the sequence $(i,\dots ,i)$
with $n$ terms. We write $t|n$ to denote the sequence
consisting of the first $n$ terms of a given sequence $t$. If $t|n=s$,
then we say that $t$ is an \emph{extension} of $s$ and we write $s\prec t$. Concatenation of two sequences $t$ and $s$, is denoted by $t\symbol{94}%
s$. 

Given any set $C \subset \R$, every bounded component of 
the set $\R \setminus C$ is called a \emph{gap} of $C$. A component of $C$ is called proper if it is not a singleton.

A perfect set $E \subset \R$ is called a \emph{Cantorval} (or an M-Cantorval) if it has infinitely many gaps and both endpoints of any gap are accumulated by gaps and proper components of $E$ (see \cite{MO}). 

Cantor sets appear in many publications in various settings.
One of research areas concerning Cantor sets is related to examining how does the algebraic difference (or sum, or product) looks like (e.g. \cite{Ta}, \cite{T19}, \cite{K1}, \cite{AC}, \cite{GJXZ}). One of the most known theorems in this topic is the classical result of Steinhaus (\cite{St}), which states that the difference set of the classical Cantor set is $[-1,1]$. Later this result was generalized several times (see \cite{K}, \cite{AI}, \cite{FN}, \cite{Now}, \cite{ACI}). Also,
there are results on the possible form of algebraic difference of Cantor sets. It occurs that (for various types of Cantor sets) the algebraic difference of Cantor sets can have one of the following forms: it can be a finite union of closed intervals, a Cantor set or a Cantorval (see \cite{AI}, \cite{MO}, \cite{Now}). Actually, there are three main types of Cantorvals that can be obtained. They are called: L-Cantorvals, M-Cantorvals and R-Cantorvals (compare \cite{MO}). In the main part of this paper, when we write about a Cantorval, we always mean an M-Cantorval.

One of the most popular type of Cantor sets is a central Cantor set (e.g. \cite{BFN}, \cite{Zaj}, \cite{Tom}, \cite{AAV}). Its construction is presented in the next section.

In our paper we will be interested in properties of the difference
set $C(a)-C(a)$. This difference was examined by many authors (e.g. \cite{AI}, \cite{K}, \cite{FF}, \cite{S}, \cite{FN}). The research in this area were often based on the relationship between central Cantor sets
and the achievement sets, that is, the sets of all possible subsums of convergent series. 
Using this theory, Anisca and Ilie in
\cite{AI} showed that a finite sum of central Cantor sets can be either a
Cantor set or a finite union of closed intervals, or a Cantorval (compare
Theorem \ref{AI}). From another result of this paper it can be inferred that $C((a_n))-C((a_n))$ is a finite union of
closed intervals if and only if $a _{n}\leq \frac{1}{3}$ for all but finitely many terms (see Theorem \ref{tw1}). 
On the other hand, Sanammi showed that if $a _{n}> \frac{1}{3}$ for all $n \in \N$, then $C((a_n))-C((a_n))$ is a Cantor set. These two results generalize the earlier result obtained by Kraft, which concerned some specific central Cantor sets (namely, where the sequence $(a_n)$ was constant).

A main goal of our paper is to find some new
conditions which imply that the difference set $C\left( a \right)
-C\left( a\right) $ is a Cantorval.
In \cite{BBFS} and \cite{BGM} there are examples of central Cantor sets, for which the difference set is a Cantorval (however, given as some achievement sets). In \cite{FN} there was given some specific sufficient condition for the difference set of a central Cantor set to be a Cantorval. 
In this paper we will find another condition, which implies that the set $C(a)-C(a)$ is a Cantorval. This result is completely different than the one proved in \cite{FN}. However, methods used in both papers are similar.

\section{Preliminaries}
In the beginning, we present the construction of a central Cantor set.
Let $a = (a_n) \in (0,1)^{\N}$. Put $I := [0, 1]$. In the first step of the construction, we delete from $I$ the
open interval $P$ centered at $\frac{1}{2}$ of length $a_1$. By $I_0$ and $I_1$ we denote the left and the right components of $I\setminus P$, respectively. Their common length is denoted by $d_1$. Now, suppose that for some $n \in \N$ and every $t \in \{0,1\}^n$ we have constructed 
the interval $I_{t}$ of length $d_{n}$. Denote by $P_{t}$ the open interval of length $a_{n+1} d_{n}$, concentric with $I_{t}$. Then let $I_{t\ha 0}$ and $I_{t\ha 1}$
be the left and the right components of the set 
$I_{t} \setminus P_{t}$, respectively. 
By $d_{n+1}$ denote the common length
of these components. 

For each $n \in \mathbb{N}$ put
\begin{equation*}
\mathcal{I}_n:=\{I_{t_1, \dots ,t_n}\colon (t_1,\dots
,t_n)\in\{0,1\}^n\}\quad\mbox{ and } \quad C_n(a):= \bigcup \mathcal{I}_n.
\end{equation*}
Finally, set $C(a): = \bigcap_{n\in\mathbb{N}}C_n(a)$. Every such a set $C(a)$ is called a 
central Cantor set.
From the construction it easily follows that the length of each interval $%
I_{t}$, where $t\in\{0,1\}^n, n \in \N$ is equal to 
\begin{equation*}
d_{n}=\prod_{i=1}^n \frac{1-a_i}{2}.
\end{equation*}%

We will now focus on examining the set $C(a)-C(a)$. We will start from the description of the set $C_n(a)-C_n(a)$. We have obviously
$$C_n(a)-C_n(a) = \bigcup_{I,I' \in \mathcal{I}_n} (I-I').$$
Now, let us introduce the family of intervals $J_s$ indexed by finite sequences $s=(s_1,\dots, s_n)$, 
$n\in\N$, with
$s_i\in\{0,1,2\}$, which depends on the sequence $a$. Very similar intervals were defined firstly in \cite{S}. They were also used in \cite{FN} and \cite{Now}.
Put $J:=I-I=\left[ -1,1\right] $. For  $n\in \mathbb{N}$ and $s\in
\left\{ 0,1,2\right\} ^{n}$ we define the interval $J_{s}$ by $%
J_{s}:=I_{t}-I_{p}$, where $p,t\in \left\{ 0,1\right\} ^{n}$ satisfy the equation $%
s_{i}=t_{i}-p_{i}+1$ for $i=1,\ldots ,n$. This definition does not depend on the choice of $p$ and $t$, because the following equalities hold (see \cite{FN})
\begin{eqnarray*}
J_{s\symbol{94}0} &=&I_{t\ha 0} - I_{p \ha 1} = \left[ l(J_{s}),l(J_{s})+2d_{n+1}\right] , \\
J_{s\symbol{94}1} &=&I_{t\ha 0} - I_{p \ha 0} =I_{t\ha 1} - I_{p \ha 1} =[c(J_{s})-d_{n+1},c(J_{s})+d_{n+1}], \\
J_{s\symbol{94}2}&=&I_{t\ha 1} - I_{p \ha 0} =[r(J_{s})-2d_{n+1},r(J_{s})].
\end{eqnarray*}%
Thus, we have
$$C_n(a)-C_n(a) = \bigcup_{s\in\{0,1,2\}^n}J_s.$$

Let $n\in \mathbb{N}$, $s\in \left\{ 0,1,2\right\} ^{n}$ and $a \in
\left( 0,1\right) ^{\mathbb{N}}$. If $a_{n+1}>\frac{1}{3}$,
then the set $J_{s}\setminus \left( J_{s\symbol{94}0}\cup J_{s\symbol{94}%
1}\cup J_{s\symbol{94}2}\right) $ is a union of two open intervals. We
denote them by $G_{s}^{0}$ and $G_{s}^{1}$, and call the \emph{left} and the 
\emph{right gap} in $J_{s}$. If $a_{n+1}\leq \frac{1}{3}$, then $J_{s%
\symbol{94}0}\cap J_{s\symbol{94}1}\neq \emptyset $ and $J_{s\symbol{94}%
1}\cap J_{s\symbol{94}2}\neq \emptyset $. We denote these intervals by $%
Z_{s}^{0}$ and $Z_{s}^{1}$, and we call them the \emph{left} and the \emph{%
right overlap} in $J_{s}$. We also assume that $d_{0}:=\left\vert
I\right\vert =1$ and that $\left\{ 0,1,2\right\} ^{0}$ is equal to the
empty sequence $s=\emptyset $. 

The following useful proposition was proved in \cite{FN}. There a sequence $ \lambda = (\lambda_n) \in (0, \frac{1}{2})^\N$ was used instead of $a$, but there is an easy connection between these two sequences. Namely, $\lambda_n = \frac{1-a_n}{2}$ for any $n \in \N$ and $\lambda_n < \frac{1}{3}$ is equivalent to $a_n > \frac{1}{3}.$
\begin{proposition}
\label{lem}Let $n,k\in \mathbb{N}$, $s,u\in \left\{ 0,1,2\right\} ^{n}$, and $%
a =(a_i) \in \left( 0,1\right) ^{\mathbb{N}}$. The following
properties hold.

\begin{enumerate}
\item \label{1-1}$\left\vert J_{s}\right\vert =2d_{n}.$

\item \label{1-2}$l\left( J_{s\symbol{94}0^{\left( k\right) }}\right)
=l\left( J_{s}\right) ,c\left( J_{s\symbol{94}1^{\left( k\right) }}\right)
=c\left( J_{s}\right) $, and $r\left( J_{s\symbol{94}2^{\left( k\right)
}}\right) =r\left( J_{s}\right) .$

\item \label{1-2b}If $n>k$, then $d_{n-1}-d_{n}<d_{k-1}-d_{k}$.

\item \label{1-2a}$l\left( J_{s}\right) -l\left( J_{u}\right) =c\left(
J_{s}\right) -c\left( J_{u}\right) =r\left( J_{s}\right) -r\left(
J_{u}\right) =\sum_{r=1}^{n}\left( s_{r}-u_{r}\right) \cdot \left(
d_{r-1}-d_{r}\right) .$

\item \label{1-3}If $a_{n+1}>\frac{1}{3}$, then $G_{s}^{0}=\left(
r\left( J_{s\symbol{94}0}\right) ,l\left( J_{s\symbol{94}1}\right) \right) $%
, $G_{s}^{1}=\left( r\left( J_{s\symbol{94}1}\right) ,l\left( J_{s\symbol{94}%
2}\right) \right) $, and $\left\vert G_{s}^{0}\right\vert =\left\vert
G_{s}^{1}\right\vert =d_{n}-3d_{n+1}.$

\item \label{1-4}If $a_{n+1}\leq \frac{1}{3}$, then $Z_{s}^{0}=\left[
l\left( J_{s\symbol{94}1}\right) ,r\left( J_{s\symbol{94}0}\right) \right] $%
, $Z_{s}^{1}=\left[ l\left( J_{s\symbol{94}2}\right) ,r\left( J_{s\symbol{94}%
1}\right) \right] $, and $\left\vert Z_{s}^{0}\right\vert =\left\vert
Z_{s}^{1}\right\vert =3d_{n+1}-d_{n}.$


\item \label{1-6}If $a_{n+1}\leq \frac{1}{3},\ldots ,
a_{n+k}\leq \frac{1}{3}$, then $J_{s}=\bigcup_{t\in \left\{ 0,1,2\right\}
^{n+k},s\prec t}J_{t}$ and $C_{n}\left( a \right) -C_{n}\left( a
\right) =C_{n+k}\left( a \right) -C_{n+k}\left( a \right) $.

\item \label{1-7}$C(a )-C(a )=\bigcap_{n\in \mathbb{N}%
}(C_{n}(a)-C_{n}(a))$.

\item \label{1-8}$C(a )+C(a )=\left( C(a )-C(a
)\right) +1$.
\end{enumerate}
\end{proposition}
We will also need the following results.
\begin{theorem} \cite{AI}
\label{AI} For any sequence $a \in \left( 0,1\right) ^{%
\mathbb{N}}$, the set $C\left(a \right) -C\left(a \right) $
has one of the following fashions:

\begin{enumerate}
\item a finite union of closed intervals;

\item a Cantor set;

\item a Cantorval.
\end{enumerate}
\end{theorem}

\begin{theorem} [see \cite{AI}, \cite{Now}, \cite{FN}]
\label{tw1} Let $a =\left( a _{n}\right) \in \left( 0,1%
\right) ^{\mathbb{N}}$. The following statements hold.

\begin{enumerate}
\item \label{tw1-1}$C\left( a \right) -C\left( a \right) =\left[
-1,1\right] $ if and only if $a _{n}\leq \frac{1}{3}$ for all $n\in 
\mathbb{N}$.

\item \label{tw1-2}$C\left( a \right) -C\left( a \right) $ is a
finite union of intervals if and only if the set $\left\{ n\in \mathbb{N}%
:a_{n}>\frac{1}{3}\right\} $ is finite.
\end{enumerate}
\end{theorem}
Fix a sequence $a =(a_n) \in (0,1)^{\N}$ such that the set $\{n \in \N\colon a_n > \frac{1}{3}\}$ is infinite and assume that there is $k_0 \in \N \cup \{0\}$ such that $a_{k_0+1} < \frac{1}{3}$. Let $(k_n)$ be an increasing sequence of all indices greater than $k_0$ such that $a_{k_n} > \frac{1}{3}$. Above notation will be used for the rest of this section. 

For $n\in \N$ and $s \in \{0,1,2\}^{k_n-1}$ the sets $G^0_s$ and $G^1_s$ are gaps in $J_s$. We call them gaps of rank $n$. We define also $N(0,s)= \max \{j \colon s_j > 0\}$ and $N(1,s)= \max \{j \colon s_j < 2\}$, where $\max \emptyset :=0$. Therefore, for $N=N(0,s)$ we have $s_N > 0$ and $s_{N+1} = \dots = s_{k_n-1} = 0$, and for $N=N(1,s)$ we have
$s_N < 2$ and $s_{N+1} = \dots = s_{k_n-1} = 2$.

From Proposition \ref{lem} (\ref{1-7}) it follows that gaps in the sets $C_n(a)-C_n(a)$ appear also in $C(a)-C(a)$. We will now define some particular gaps in intervals $J_t$ which will be important in our further considerations.

Take $k \in \N$ and $t \in \{0,1,2\}^k$. Let $n \in \N$ be such that $k_n > k$. If $s \in \{0,1,2\}^{k_n-1}$, $t \prec s$ and $i \in \{0,1\}$, then we say that a gap $G_s^i$ is a gap \emph{from} $J_t$ (or comes from $J_t$).
The gap $G^0_{t \ha 0^{(k_n-k-1})}$ is called the leftmost gap in $J_t$ of rank $n$ and is denoted by $\bar{G}^0_t(n)$. It is a gap which comes from $J_t$ and is nearest the left endpoint of $J_t$. Analogously, the gap $G^1_{t \ha 2^{(k_n-k-1})}$ is called the rightmost gap in $J_t$ of rank $n$ and is denoted by $\bar{G}^1_t(n)$. It is a gap which comes from $J_t$ and is nearest the right endpoint of $J_t$.

Take $m \in \N$ and $s \in \{0,1,2\}^{k_m-1}$. 
For $n\geq m$ we define inductively the families of gaps of rank $n$ $\G ^0 _s(n)$ and $\G^1_s(n)$. For $i \in \{0,1\}$ put $\G ^i_s(m) := \{G^i_s\}$. 

Assume that for some $n \geq m$ we have defined the families $\G ^i_s(l)$ for all $ l\in \{m, m+1, \dots, n\}$ and $i \in \{0,1\}$. We define
$$\G ^0_s(n+1):= \{\bar{G}^0_{s}(n+1)\} \cup \bigcup_{l \leq n} \left( \bigcup_{t:G_t^0 \in \G^0_s(l)} \{\bar{G}^1_{t\,\hat{}\,0}(n+1), \bar{G}^0_{t\,\hat{}\,1}(n+1)\} \cup \bigcup_{t:G_t^1 \in \G^0_s(l)} \{\bar{G}^1_{t\,\hat{}\,1}(n+1), \bar{G}^0_{t\,\hat{}\,2}(n+1)\}\right) $$
$$=
\{\bar{G}^0_{s}(n+1)\} \cup \bigcup_{l \leq n}\, \bigcup_{i \in \{0,1\}} \,\bigcup_{t:G_t^i \in \G^0_s(l)} \{\bar{G}^1_{t\,\hat{}\,i}(n+1), \bar{G}^0_{t\,\hat{}\,(i+1)}(n+1)\} 
$$
and
$$\G ^1_s(n+1):= \{\bar{G}^1_{s}(n+1)\} \cup \bigcup_{l \leq n} \left( \bigcup_{t:G_t^0 \in \G^1_s(l)} \{\bar{G}^1_{t\,\hat{}\,0}(n+1), \bar{G}^0_{t\,\hat{}\,1}(n+1)\} \cup \bigcup_{t:G_t^1 \in \G^1_s(l)} \{\bar{G}^1_{t\,\hat{}\,1}(n+1), \bar{G}^0_{t\,\hat{}\,2}(n+1)\}\right)$$
$$=
\{\bar{G}^1_{s}(n+1)\} \cup \bigcup_{l \leq n}\, \bigcup_{i \in \{0,1\}} \,\bigcup_{t:G_t^i \in \G^1_s(l)} \{\bar{G}^1_{t\,\hat{}\,i}(n+1), \bar{G}^0_{t\,\hat{}\,(i+1)}(n+1)\}.
$$
In particular,
$$\G^0_s(m+1) = \{ \bar{G}^0_{s}(m+1),\bar{G}^0_{s\ha 1}(m+1), \bar{G}^1_{s\ha 0}(m+1)\}$$ 
and
$$\G^1_s(m+1) = \{ \bar{G}^1_{s}(m+1),\bar{G}^0_{s\ha 2}(m+1), \bar{G}^1_{s\ha 1}(m+1)\}.$$

Geometrically, the family $\G^0_s(n)$ consist of the leftmost gap of rank $n$ in $J_s$ and those gaps from $J_s$ which are nearest both endpoints of all gaps  from the families $\G^0_s(l)$ for $l < n$. The families $\G^1_s(n)$ can be described analogously.

We have defined the families $\G ^0 _s(n)$ and $\G^1_s(n)$ for $n \geq m$, where $s \in \{0,1,2\}^{k_m-1}$. Now, we define unions of these gaps, putting $$\G ^i_s := \bigcup_{n \geq m} \G ^i_s (n)$$ for $i \in \{0,1\}.$ 

Using the definitions above, we define a particular family of gaps in any interval of the form $J_t$. Namely, let $k\geq k_{0}$ and $t\in \left\{ 0,1,2\right\} ^{k}$. Then there is $m\in \mathbb{N}$ such that $k_{m-1}\leq k<k_{m}$. Set
$$\G_t:=\G^0_{t\ha 0^{(k_m-k-1)}} \cup \G^1_{t\ha 2^{(k_m-k-1)}} = \bigcup_{n \geq m} \left(\G^0_{t\ha 0^{(k_m-k-1)}}(n) \cup \G^1_{t\ha 2^{(k_m-k-1)}}(n) \right).$$
It is worth mentioning that the family $\G_t$ defined above is exactly the same as the family $\G_t$ defined in \cite{FN}.

The next theorem proved in \cite{FN} shows, under some not restrictive assumptions, that a set of the form $J_t \setminus \bigcup \G_t$ has nonempty interior. We will use this fact later in the proof of Theorem \ref{main}. We will show, under some stronger assumptions, that $J_t \setminus \bigcup \G_t \subset C(a)-C(a)$, and hence $C(a)-C(a)$ has nonempty interior.
\begin{theorem} \cite{FN}
\label{tw4}Assume that $a =\left( a _{j}\right) _{j\in \mathbb{N}%
}\in \left( 0,1\right) ^{\mathbb{N}}$ is a sequence such that: $%
a _{n}>\frac{1}{3}$\ for infinitely many terms, $a _{n}\leq
\frac{1}{3}$\ for infinitely many terms, and $k_{0}\in \mathbb{N%
}\cup \left\{ 0\right\} $ is such that $a_{k_{0}+1}<\frac{1}{3}$. Let $%
k\geq k_{0}$, $t\in \left\{ 0,1,2\right\} ^{k}$ and $m\in \mathbb{N}$ be
such that $k_{m-1}\leq k<k_{m}$, where the sequence $\left( k_n \right)$ consists of all indices greater than $k_0$, for which $a_{k_n} > \frac{1}{3}$. The following statements hold.

\begin{enumerate}
\item \label{tw4-1}The set $J_{t}\setminus \bigcup \G_t $ has
nonempty interior.

\item \label{tw4-2}If $k=k_{0}=0$ and $t=\emptyset $, then $C\left( a
\right) -C\left( a \right) \subset \left[ -1,1\right] \setminus
\G_t $ and 
\begin{equation*}
\left\vert \bigcup \G_t \right\vert =\sum_{n=1}^{\infty }2\cdot
3^{n-1}\left( d_{k_{n}-1}-3d_{k_{n}}\right) .
\end{equation*}
\end{enumerate}
\end{theorem}

Let $m \in \N$, $s, u \in \{ 0,1,2\} ^{k_m-1}$. We say that the intervals $J_s$ and $J_u$ are \emph{associated} if $N(0,s) = N(1,u)\in \{ k_{m-1}+1, \dots ,k_{m}-1\} $ and for $N:=N(0,s)=N(1,u)$ the equalities $s|(N-1) = u|(N-1)$ and $u_N = s_N -1$ hold. We then say that $J_u$ is an interval covering $\G ^0_s$ and $J_s$ is an interval covering $\G^1_u$. 
Also, for any $G \in \G^0_s$ we say that $J_u$ is an interval covering $G$, and we write $G \subset_* J_u$. Similarly, for any $H \in \G^1_u$ we say that $J_s$ is an interval covering $H$, and we write $H \subset_* J_s$.

%

In the sequel we will show that, under some additional assumptions, if an interval $J_u$ is the interval covering $\G^i_s$, then every gap $G \in \G^i_s$ is included in the set $\bigcap_{n\in \N} \bigcup_{t\in \{0,1,2\}^n} J_{u\ha t}$, that is, it is covered by a part of the set $C(a)-C(a)$ which comes from $J_u$.  
%

\begin{lemma} \label{uw}
Let $k\geq k_{0}$, $t\in \left\{ 0,1,2\right\} ^{k}$ and $m\in \mathbb{N}$ be such that $k_{m-1}\leq k<k_{m}$. Let $n \geq m$, $t \prec s \in \{0,1,2\}^{k_n-1}$, $i \in \{0,1\}$. If $G^i_s \notin \G_t$, then $N(i,s) > k$. \end{lemma}
\begin{proof}
Suppose that $N=N(i,s) \leq k$. Since $t = s|k$, we have $$s = (s|N) \ha (2i)^{(k_n-1-N)} = t \ha (2i)^{(k_n-1-k)},$$ so $G^i_s \in \G^i_{t\ha (2i)^{(k_m-k-1)}} \subset \G_t.$ 
\end{proof}
Let $s \in \{0,1,2\}^k$, $k \in \{k_{m-1}, \dots, k_m-1\}, m \in \N$. For $n \geq m$ define gaps of rank $n$:

$$G^0(s,n) := G^0_{s \ha 0^{(k_m-k-1)} \ha 1 \ha 0^{(k_{m+1}-k_m -1)} \ha 1 \ha \dots \ha 1 \ha 0^{(k_n-k_{n-1}-1)}},$$
$$G^1(s,n): = G^1_{s \ha 2^{(k_m-k-1)} \ha 1 \ha 2^{(k_{m+1}-k_m -1)} \ha 1 \ha \dots \ha 1 \ha 2^{(k_n-k_{n-1}-1)}}.$$
In particular,
$$G^0(s,m) = \bar{G}^0_s(m)$$
$$G^1(s,m) = \bar{G}^1_{s}(m).$$

\begin{lemma} \label{mon}
Let $m \in \N$, $k \in \{k_{m-1}, \dots, k_m-1\}$ and $s \in \{0,1,2\}^k$. The sequences $(r(G^0(s,n)))_{n \geq m}$ and $(l(G^0(s,n)))_{n \geq m}$ are increasing, and the sequences $(l(G^1(s,n)))_{n \geq m}$ and $(r(G^1(s,n)))_{n \geq m}$ are decreasing. Moreover, for $n \geq m$, \begin{equation} \label{r}
r(G^0(s,n+1))-r(G^0(s,n)) = l(G^1(s,n)) - l(G^1(s,n+1)) = d_{k_{n+1}-1}-d_{k_{n+1}}
\end{equation} and 
\begin{equation}\label{lG}
l(G^1(s,n)) = r(J_s) -\sum_{i=m}^n (d_{k_i-1}-d_{k_{i}}),
\end{equation} 
\begin{equation}\label{rG}
r(G^0(s,n)) = l(J_s) + \sum_{i=m}^n (d_{k_i-1}-d_{k_{i}}).
\end{equation}
\end{lemma}
\begin{proof}
We will show that the sequences $(l(G^1(s,n)))$ and $(r(G^1(s,n)))$ are decreasing. Put $t = s \ha 2^{(k_{m}-k-1)} \ha 1 \dots \ha 1 \ha 2^{(k_n-k_{n-1}-1)}.$ By Proposition \ref{lem} (\ref{1-2}), (\ref{1-3}) and (\ref{1-2a}), we have
$$l(G^1(s, n+1))-l(G^1(s,n)) = l(G^1_{t \ha 1 \ha 2^{(k_{n+1}-k_{n}-1)}}) - l(G^1_t)$$
$$= r(J_{t \ha 1 \ha 2^{(k_{n+1}-k_{n}-1)} \ha 1}) - r(J_{t \ha 1 })= r(J_{t \ha 1 \ha 2^{(k_{n+1}-k_{n}-1)} \ha 1}) - r(J_{t \ha1 \ha 2^{(k_{n+1}-k_{n}-1)}\ha 2}) = d_{k_{n+1}}-d_{k_{n+1}-1} < 0$$
and
$$r(G^1(s, n+1))-r(G^1(s, n)) = r(G^1_{t \ha 1 \ha 2^{(k_{n+1}-k_{n}-1)}}) - r(G^1_t)< r(J_{t \ha 1}) - r(G^1_t)=l(G^1_t)-r(G^1_t)<0.$$
Therefore, the sequences $(l(G^1(s,n)))$ and $(r(G^1(s,n)))$ are decreasing. 
Similarly, for \\$v = s \ha 0^{(k_{m}-k-1)} \ha 1 \dots \ha 1 \ha 0^{(k_n-k_{n-1}-1)}$, $$ r(G^0(s, n+1)) - r(G^0(s, n)) = l(J_{v \ha 1 \ha 0^{(k_{n+1}-k_n-1)} \ha 1}) - l(J_{v \ha 1 \ha 0^{(k_{n+1}-k_n-1)}\ha 0})= d_{k_{n+1}-1}-d_{k_{n+1}} >0$$
and
$$l(G^0(s, n+1)) - l(G^0(s, n)) > r(G_v^0)  - l(G^0_v)>0.$$
Hence the sequences $(r(G^0(s,n)))$ and $(l(G^0(s,n)))$ are increasing. We have also proved (\ref{r}).
Moreover, by Proposition \ref{lem} (\ref{1-2a}), (\ref{1-2}), we have $$l(G^1(s,m)) = l(G^1_{s \ha 2^{(k_{m}-k-1)}}) = r(J_{s \ha 2^{(k_{m}-k-1)} \ha 1}) = r(J_{s \ha 2^{(k_{m}-k)} }) -(d_{k_m-1}-d_{k_{m}}) = r(J_s) - d_{k_m-1}+d_{k_{m}}$$ and
$$r(G^0(s,m))=r(G^0_{s \ha 0^{(k_{m}-k-1)}}) = r(J_{s \ha 0^{(k_{m}-k-1)} \ha 1}) = l(J_{s \ha 0^{(k_{m}-k)} }) +d_{k_m-1}-d_{k_{m}} = l(J_s) +d_{k_m-1}-d_{k_{m}}.$$
So, by (\ref{r}), for $n \geq m$ we have
$$l(G^1(s,n)) = l(G^1(s,n-1)) - (d_{k_n-1}-d_{k_n}) = \dots = l(G^1(s,m)) - \sum_{i=m+1}^n (d_{k_i-1}-d_{k_{i}}) = $$$$r(J_s) - \sum_{i=m}^n (d_{k_i-1}-d_{k_{i}})$$
and 
$$r(G^0(s,n)) = r(G^0(s,m)) + \sum_{i=m+1}^n (d_{k_i-1}-d_{k_{i}}) = l(J_s) + \sum_{i=m}^n (d_{k_i-1}-d_{k_{i}}).$$

\end{proof}

\begin{lemma} \label{lemat 2}
Let $m \in \N$ and $s \in \{0,1,2\}^{k_m-1}$, $n > m$. Then
\begin{equation}\label{l1}
r(G^1(s\ha 0, n)) < l(G^0_s) <r(G^0_s) < l(G^0(s, n)),
\end{equation}
\begin{equation}\label{l2}
r(G^1(s, n)) < l(G^1_s) < r(G^1_s) < l(G^0(s \ha 2, n)).
\end{equation}
Moreover, for all $G \in \G^0_s(n)$ 
\begin{equation} \label{1}
\left[r(G) \leq r(G^0(s \ha 0,n))\right] \mbox{ or } \left[ l(G^1(s\ha 0, n)) \leq l(G)< r(G) \leq r(G^0(s, n))\right].
\end{equation}
and for all $G\in \G^1_s(n)$  
\begin{equation} \label{2}
\left[l(G) \geq l(G^1(s \ha 2,n))\right] \mbox{ or } \left[ l(G^1(s, n)) \leq l(G) < r(G) \leq r(G^0(s \ha 2, n))\right].
\end{equation}

\end{lemma}
\begin{proof}

By Proposition \ref{lem} (\ref{1-2}), (\ref{1-3}), we have
$$l(G^0_s) - r(G^1(s\ha 0, m+1)) = r(J_{s \ha 0})-r(G^1_{s \ha 0 \ha 2^{(k_{m+1}-k_m-1)}})= r(J_{s \ha 0})- l(J_{s \ha 0 \ha 2^{(k_{m+1}-k_m)}}) $$$$= r(J_{s \ha 0 \ha 2^{(k_{m+1}-k_m)}}) - l(J_{s \ha 0 \ha 2^{(k_{m+1}-k_m)}}) = |J_{s \ha 0 \ha 2^{(k_{m+1}-k_m)}}| > 0$$
and
$$l(G^0(s,m+1)) -r(G^0_s) =l(G^0_{s\ha 1 \ha 0^{(k_{m+1}-k_m-1)}}) - l(J_{s\ha 1}) =r(J_{s\ha 1 \ha 0^{(k_{m+1}-k_m)}}) - l(J_{s\ha 1}) > l(J_{s\ha 1}) - l(J_{s \ha 1}) = 0.$$
By Lemma \ref{mon}, the sequence $(r(G^1(s\ha 0, n)))$ is decreasing, and the sequence $(l(G^0(s, n)))$ is increasing. Hence
$$r(G^1(s\ha 0, n))\leq r(G^1(s\ha 0, m+1)) < l(G^0_s) < r(G^0_s) <l(G^0(s,m+1))  \leq l(G^0(s, n)), $$
which proves (\ref{l1}).
The proof of (\ref{l2}) is analogous.

Now, we will inductively show that for any $n > m$, if $G \in \G^0_s(n),$ then (\ref{1}) holds. 
Suppose that $n = m+1$. 
By the definition,
$$\G^0_s(m+1) = \{\bar{G}^0_s(m+1), \bar{G}^1_{s\ha 0}(m+1), \bar{G}^0_{s\ha 1}(m+1) \},$$ 
so there are three possible cases.

1. $G =\bar{G}^0_s(m+1)= G^0_{s \ha 0^{(k_{m+1}-k_m-1)}}$. Then $G = G^0(s \ha 0,n)$, so the first condition from (\ref{1}) is satisfied.

2. $G = \bar{G}^1_{s\ha 0}(m+1)= G^1_{s \ha 0 \ha 2^{(k_{m+1}-k_m-1)}}$. Then $G = G^1(s \ha 0,n)$. Since, by (\ref{l1}), $$l(G^1(s\ha 0,n))=l(G) <r(G)=r(G^1(s\ha 0,n)) < l(G^0_s) < r(G^0_s) <l(G^0(s,n))< r(G^0(s,n)),$$ the second condition from (\ref{1}) is satisfied. 

3. $G =\bar{G}^0_{s\ha 1}(m+1)= G^0_{s \ha 1 \ha 0^{(k_{m+1}-k_m-1)}}$. Then $G = G^0(s,n)$. Since, by (\ref{l1}), $$l(G^1(s\ha 0,n))<r(G^1(s\ha 0,n))< l(G^0_s) < r(G^0_s) <l(G^0(s,n))=l(G)<  r(G)= r(G^0(s,n)),$$ the second condition from (\ref{1}) is satisfied. 

Let $n > m$. Assume that for all $l \leq n$ and all $G \in \G^0_s(l)$, (\ref{1}) holds. We will show that for all $G \in \G^0_s(n+1)$ the condition (\ref{1}) is satisfied too. By the definition, 
$$\G^0_s(n+1)=
\{\bar{G}^0_{s}(n+1)\} \cup \bigcup_{l \leq n}\, \bigcup_{i \in \{0,1\}} \,\bigcup_{t:G_t^i \in \G^0_s(l)} \{\bar{G}^1_{t\,\hat{}\,i}(n+1), \bar{G}^0_{t\,\hat{}\,(i+1)}(n+1)\}, 
$$ so there are three possible cases.

1. $G = \bar{G}^0_{s}(n+1)=G^0_{v}$, where $v=s \ha 0^{(k_{n+1}-k_m-1)}$. We have $G^0(s \ha 0,n+1) = G^0_t$, where
$$t = s \ha 0\ha 0^{(k_{m+1}-k_m-1)} \ha 1 \ha \dots \ha 1 \ha 0^{(k_{n+1}-k_{n}-1)}.$$ 
By Proposition \ref{lem} (\ref{1-2a}), $l(J_t) \geq l(J_v)$, and thus, by Proposition \ref{lem} (\ref{1-3}), $r(G^0_v)=l(J_{v \ha 1}) \leq l(J_{t \ha 1}) =r(G^0_t).$ Hence the first condition in (\ref{1}) is satisfied.

2. $G = \bar{G}^1_{t\,\hat{}\,i}(n+1)= G^1_{t \ha i \ha 2^{(k_{n+1}-k_l-1)}}$ for some $l \geq m$, $t \in \{0,1,2\}^{k_l-1}$ and $i \in \{0,1\}$ such that $G^i_t \in \G^0_s(l)$. By the induction hypothesis, $r(G^i_t) \leq r(G^0(s \ha 0,l))$ or $ l(G^1(s\ha 0, l) \leq l(G^i_t) < r(G^i_t) \leq r(G^0(s, l)).$
We also have, by Proposition \ref{lem} (\ref{1-3}), $r(G^i_t) > l(G^i_t) = r(J_{t \ha i}) > r(G).$
Consider two subcases.

2.1. $r(G^i_t) \leq r(G^0(s \ha 0,l))$. Then, by Lemma \ref{mon}, we have $$r(G^0(s\ha 0,n+1)) \geq r(G^0(s\ha 0,l)) \geq r(G^i_t) > r(G),$$
so the first condition in (\ref{1}) is satisfied.

2.2. $ l(G^1(s\ha 0, l)) \leq l(G^i_t) < r(G^i_t) \leq r(G^0(s, l))$. By Lemma \ref{mon}, $$r(G^0(s, n+1))\geq r(G^0(s, l)) \geq r(G^i_t) >  r(G).$$ By Proposition \ref{lem} (\ref{1-3}), (\ref{1-2a}), (\ref{1-2}) and Lemma \ref{mon}, we have
$$l(G) =l(G^1_{t \ha i \ha 2^{(k_{n+1}-k_l-1)}})= r(J_{t \ha i \ha 2^{(k_{n+1}-k_l-1)} \ha 1}) = r(J_{t \ha i \ha 2^{(k_{n+1}-k_l)}}) - d_{k_{n}}+d_{k_{n+1}} = r(J_{t \ha i}) - d_{k_{n}}+d_{k_{n+1}} $$$$= l(G^i_t)- d_{k_{n}}+d_{k_{n+1}} \geq l(G^1(s\ha 0, l)) - d_{k_{n}}+d_{k_{n+1}} \geq l(G^1(s\ha 0, n)) - d_{k_{n}}+d_{k_{n+1}} \stackrel{(\ref{lG})}{=} l(G^1(s\ha 0, n+1)),$$
so the second condition in (\ref{1}) is satisfied.

3. $G = \bar{G}^0_{t\,\hat{}\,(i+1)}(n+1) = G^0_{t \ha (i+1) \ha 0^{(k_{n+1}-k_l-1)}}$ for some $l \geq m$, $t \in \{0,1,2\}^{k_l-1}$ and $i \in \{0,1\}$ such that $G^i_t \in \G^0_s(l)$. By the induction hypothesis, $r(G^i_t) \leq r(G^0(s \ha 0,l))$ or $l(G^1(s\ha 0, l)) \leq l(G^i_t) < r(G^i_t) \leq r(G^0(s, l))).$ Consider two subcases.

3.1. $r(G^i_t) \leq r(G^0(s \ha 0,l))$. Then, using Proposition \ref{lem} (\ref{1-2}), (\ref{1-2a}), (\ref{1-3}) and Lemma \ref{mon}, we obtain
$$r(G) = l(J_{t \ha (i+1) \ha 0^{(k_{n+1}-k_l-1)} \ha 1}) = l(J_{t \ha (i+1) \ha 0^{(k_{n+1}-k_l-1)}}) + d_{k_{n}}-d_{k_{n+1}} = l(J_{t \ha (i+1)}) + d_{k_{n}}-d_{k_{n+1}} $$$$= r(G^i_t) +d_{k_{n}}-d_{k_{n+1}} \leq r(G^0(s \ha 0,l)) + d_{k_{n}}-d_{k_{n+1}} \leq r(G^0(s \ha 0,n)) + d_{k_{n}}-d_{k_{n+1}} \stackrel{(\ref{rG})}{=} r(G^0(s \ha 0,n+1)),$$
so the first condition in (\ref{1}) is satisfied. 

3.2. $l(G^1(s\ha 0, l)) \leq l(G^i_t) < r(G^i_t) \leq r(G^0(s, l)))$. By Proposition \ref{lem} (\ref{1-3}), we have $$l(G^1(s\ha 0, l)) \leq l(G^i_t) <r(G^i_t) = l(J_{t \ha (i+1)}) \leq l(G).$$  Using calculations from 3.1., 
Lemma \ref{mon} 
and the assumption, we get
$$r(G) = r(G^i_t) + d_{k_{n}}-d_{k_{n+1}} \leq r(G^0(s, l)) + d_{k_{n}}-d_{k_{n+1}} \leq r(G^0(s, n)) + d_{k_{n}}-d_{k_{n+1}} = r(G^0(s, n+1)),$$
so the second condition in (\ref{1}) is satisfied.

We have proved that if $G \in \G^0_s(n+1),$ then the condition (\ref{1}) is satisfied. By the induction, we have the assertion.

The proof that (\ref{2}) is satisfied for any $G \in \G^1_s(n)$ is similar.
\end{proof}

\section{Main results}
Now, we can prove the main theorem of the paper.
\begin{theorem} \label{main}
Let $a = (a_n) \in (0,1)^{\N}$. Assume that the set $\{n \in \N \colon a_n > \frac{1}{3}\}$ is infinite. Let $k_0$ be such that $a_{k_0+1} < \frac{1}{3}.$ Let the sequence $\left( k_n \right)$ consist of all indices greater than $k_0$, for which $a_{k_n} > \frac{1}{3}$. Let $k\geq k_{0}$, $t\in \left\{ 0,1,2\right\} ^{k}$ and $m\in \mathbb{N}$ be such that $k_{m-1}\leq k<k_{m}$. Denote
$$ \delta _n := \min \{ 3d_i - d_{i-1} \colon i \in \{k_{n-1}+1, \dots, k_n-1\}\},$$
$$ \Delta _n := \max \{ 3d_i - d_{i-1} \colon i \in \{k_{n-1}+1, \dots, k_n-1\}\},$$
where $\max \emptyset =-\infty, \min \emptyset =\infty$.
Put
$$m_n := \min \{\delta_{n}-(d_{k_{n}-1}-d_{k_{n}}), 4d_{k_{n}}-\Delta_{n} \} $$ for $n \in \N$.
If for any $n\in\N$ we have
$$(*) \,\,\, m_n \geq 2\cdot\sum_{i=n+1}^{\infty}(d_{k_i-1}-d_{k_i}), $$
then the set $C(a) - C(a)$ is a Cantorval. Moreover, if $k=k_{0}=0$ and $t=\emptyset $, then $C\left( a
\right) -C\left( a \right) =J_{t}\setminus \bigcup \G_t $
and
\begin{equation*}
\left\vert C\left(a \right) -C\left(a \right) \right\vert
=2-2\sum_{n=1}^{\infty }3^{n-1}\left( d_{k_{n}-1}-3d_{k_{n}}\right) .
\end{equation*}
\end{theorem}
\begin{proof}
In the beginning, we will prove

{\bf Claim 1.}
If $n \geq m$, $s, u \in \{0,1,2\}^{k_n-1}$, $t \prec s$, $G^0_s \subset_* J_u$, 
then \\$G^0_s \subset J_{u \ha 2}$, $G^1_u \subset J_{s \ha 0}$ and $l(G^0_s)-r(G^1_u) \geq m_n$,
$r(J_u)-r(G^0_s)\geq m_n$, $l(G^1_u)-l(J_s)\geq m_n$. 

By the definition of a covering interval, we know that $N= N(0,s) > k_{n-1}$ and 
$$u = (s|(N-1)) \,\hat{}\,(s_N-1) \,\hat{}\, 2^{(k_n-N-1)} \in \{0,1,2\}^{k_n-1}. $$  
From Lemma \ref{lem} (\ref{1-2}) and (\ref{1-4}) we have
$$r(J_{u}) = r(J_{(s|(N-1)) \,\hat{}\,(s_N-1)})=r(Z^{s_N-1}_{(s|(N-1))}) = l(Z^{s_N-1}_{(s|(N-1))})+|Z^{s_N-1}_{(s|(N-1))}|$$$$ = l(J_{s|N})+3d_N-d_{N-1}= l(J_s)+3d_N-d_{N-1}.$$ 
Hence, by Lemma \ref{lem} (\ref{1-3}), (\ref{1-2a}) and (\ref{1-1}),
$$r(J_{u}) - r(G^0_{s}) =l(J_{s}) + 3d_{N}-d_{N-1} - l(J_{s\,\hat{}\,1})=l(J_{s}) + 3d_{N}-d_{N-1} - (l(J_{s}) + d_{k_n-1}-d_{k_n})$$$$=3d_{N}-d_{N-1} - d_{k_n-1}+d_{k_n} \geq \delta_n - d_{k_n-1}+d_{k_n} \geq m_n $$
$$l(G^0_{s}) - r(G^1_{u}) =r(J_{s\,\hat{}\,0}) - l(J_{u\ha 2})= l(J_{s}) +2d_{k_n} - (r(J_u) -2d_{k_n})=l(J_{s}) +2d_{k_n} - (l(J_{s}) + 3d_{N}-d_{N-1} -2d_{k_n})  $$$$=4 d_{k_n} - (3d_{N}-d_{N-1}) \geq 4 d_{k_n} - \Delta_n \geq m_n $$
$$l(G^1_{u}) - l(J_{s}) =r(J_{u\,\hat{}\,1}) - l(J_s) =r(J_{u})-d_{k_n-1}+d_{k_n} - l(J_s)$$$$=3d_{N}-d_{N-1} - d_{k_n-1}+d_{k_n}\geq\delta_n - d_{k_n-1}+d_{k_n} \geq m_n$$ 

In consequence, since $m_n > 0$, we have 
$$l(J_{u\ha 2}) = r(G^1_u) < l(G^0_s) < r(G^0_s) < r(J_u) = r(J_{u\ha 2}),$$
so $G^0_s \subset J_{u \ha 2}$. Similarly, $G^1_u \subset J_{s \ha 0}$, which finishes the proof of Claim 1.

Now, we will prove inductively

{\bf Claim 2.} 
For any $n \geq m$

$(**)$ if
$t \prec s \in \{0,1,2\}^{k_n-1}$ and $i \in \{0,1\}$ are such that \begin{equation} \label{war}
G^i_s \notin \G_{t},
\end{equation}
then there exists a a finite sequence $u$ with terms in $\{0,1,2\}$ such that $t \prec u$ and $G^i_s \subset_* J_u$.

%

First, observe that $(**)$ is trivially satisfied for $n = m$, because there does not exist a gap satisfying (\ref{war}). Let $n \geq m$ and assume that $(**)$ holds for all $j=m, \dots,n$. We will show that  $(**)$ holds for $n+1$. Assume that $s \in \{0,1,2\}^{k_{n+1}-1}$, $t \prec s$ and $i \in \{0,1\}$ are such that $G^i_s \notin \G_t.$ Put $N: = N(i,s)$. By Lemma \ref{uw}, $N > k$.

Consider the cases. 

1. $N \geq k_n +1.$ 

By the definition, there exists the interval $J_u$, where $u \in \{0,1,2\}^{k_{n+1}-1},$ $t \prec u$, associated with $J_s$. $J_u$ is an interval covering $\G^i_s$, and so $G^i_s \subset_* J_u$.

2. $k_{j-1} < N < k_j$ for $m \leq j \leq n$.

Then $s = s|N \ha (2i)^{(k_{n+1}-N-1)} = s|N \ha (2i)^{(k_j-1-N)} \ha (2i)^{(k_{n+1}-k_j)} =s|(k_j-1) \ha (2i)^{(k_{n+1} - k_j)}$. Therefore, $G^i_s \in \G^i_{s|(k_j-1)}$. Moreover, $N(i,s|(k_j-1)) > k_{j-1}$, so, by the definition, there is an interval $J_u$ covering $\G^i_{s|(k_j-1)}$, where $t \prec u$, and in consequence also covering $G^i_s$.  

3.1. $N=N(0,s)=k_j$ for some $m+1 \leq j \leq n$ and $i=0$.

Consider the gap $G^{s_N-1}_{s|(k_j-1)} $. If $G^{s_N-1}_{s|(k_j-1)} $ satisfies (\ref{war}), then, by the induction hypothesis, there exists an interval $J_u$ covering it, where $t \prec u$, and so there exists a family of gaps $\G$ such that $G^{s_N-1}_{s|(k_j-1)} \in \G$ and $J_u$ is an interval covering the family $\G$. We have $s= (s|(k_j-1))\ha (s_N-1+1)\ha 0^{(k_{n+1}-k_j -1)}$, so $G^0_s \in \G$, and thus $J_u$ is an interval covering also $G^0_s$. If $G^{s_N-1}_{s|(k_j-1)}  \in \G_t =  \G^0_{t \ha 0^{(k_m-k-1)}} \cup \G^1_{t \ha 2^{(k_m-k-1)}}$, then since $s= (s|(k_j-1))\ha (s_N-1+1)\ha 0^{(k_{n+1}-k_j -1)}$, also $G^i_s \in \G^0_{t \ha 0^{(k_m-k-1)}} \cup \G^1_{t \ha 2^{(k_m-k-1)}} = \G_t$.
 
3.2. $N=N(1,s)=k_j$, for some $j \leq n$ and $i=1$.

Consider the gap $G^{s_N+1}_{s|(k_j-1)}$. If $G^{s_N-1}_{s|(k_j-1)} $ satisfies (\ref{war}), then there is an interval $J_u$ covering it, where $t \prec u$. Similarly as in 3.1., $J_u$ is an interval covering also $G^1_s$, because $s= (s|(k_j-1))\,\hat{}\,(s_N+1-1)\,\hat{}\,2^{(k_{n+1}-k_j -1)}$. If $G^{s_N+1}_{s|(k_j-1)}  \in \G_t$, then also $G^i_s \in \G_t$. This finishes the proof of Claim 2.

Now, we will prove

{\bf Claim 3.} For all $n \in \N, u \in \{0,1,2\}^{k_n-1}, s \in \{0,1,2\}^{k_n-1}$ if $J_u$ and $J_s$ are associated intervals with $J_u$ being an interval covering $\G^0_s$, then for any gaps $G \in \G^0_s, H \in \G^1_u$ we have $G \cap H = \emptyset$. If $D \in \G^1_s, F \in \G^0_u$, then also $D \cap H = \emptyset$, $G \cap F = \emptyset$ and $D \cap F = \emptyset$.

Let $n \in \N$, $u \in \{0,1,2\}^{k_n-1}, s \in \{0,1,2\}^{k_n-1}$ be such that $G^0_s \subset_* J_u$. Let $G$ be a gap of rank $j$ and $H$ a gap of rank $l$, where $j, l \geq n$, such that $G \in \G^0_s, H \in \G^1_u$. 
Without loss of generality we can assume that $j \geq l$ (the proof when $j < l$ is analogous). 
By Lemma \ref{lemat 2}, we have
$$r(G) \leq r(G^0(s \ha 0,j)) \mbox{ or } (l(G^1(s\ha 0, j)) \leq l(G) < r(G) \leq r(G^0(s, j)))$$ 
and
$$l(H) \geq l(G^1(u \ha 2,l)) \mbox{ or } ( l(G^1(u, l)) \leq l(H) < r(H) \leq r(G^0(u \ha 2, l))).$$
Consider the cases.

1. $r(G) \leq r(G^0(s \ha 0,j))$. Since, by definition and Lemma \ref{lem} (\ref{1-2a}), $l(G^1(u \ha 2,l)) \geq  l(G^1(u, l))$, we have $l(H) \geq l(G^1(u, l))$, and hence, using Lemma \ref{mon}, $(*)$, Claim 1 and Lemma \ref{lem} (\ref{1-2}), (\ref{1-3}), we obtain
$$l(H) - r(G) \geq l(G^1(u, l)) - r(G^0(s \ha 0,j)) = r(J_u) - \sum_{i=n}^l (d_{k_i-1}-d_{k_i}) - l(J_{s \ha 0}) - \sum_{i=n+1}^j (d_{k_i-1}-d_{k_i}) $$$$\geq r(J_u)- d_{k_n-1} + d_{k_n} - l(J_s) - 2 \sum_{i=n+1}^j(d_{k_i-1}-d_{k_i}) =r(J_{u\ha 1}) - l(J_s) - 2 \sum_{i=n+1}^j(d_{k_i-1}-d_{k_i})$$$$=l(G^1_u) - l(J_s) - 2 \sum_{i=n+1}^j(d_{k_i-1}-d_{k_i})  > l(G^1_u) - l(J_s) - 2 \sum_{i=n+1}^{\infty}(d_{k_i-1}-d_{k_i}) \geq l(G^1_u) - l(J_s) - m_n \geq 0.$$
Therefore, $G \cap H = \emptyset$. 

2. $l(H) \geq l(G^1(u \ha 2,l))$. Similarly as in 1., $r(G) \leq r(G^0(s, j))$, and thus
$$l(H) - r(G) \geq l(G^1(u \ha 2,l)) - r(G^0(s, j)) = r(J_u) - \sum_{i=n+1}^l (d_{k_i-1}-d_{k_i}) - l(J_{s}) - \sum_{i=n}^j (d_{k_i-1}-d_{k_i})$$$$ \geq r(J_u) - r(G^0_s) - 2 \sum_{i=n+1}^j(d_{k_i-1}-d_{k_i}) > r(J_u) - r(G^0_s) - 2 \sum_{i=n+1}^{\infty}(d_{k_i-1}-d_{k_i})\geq r(J_u) - r(G^0_s) - m_n \geq 0.$$
Hence $G \cap H = \emptyset$. 

3. $l(G) \geq l(G^1(s\ha 0, j)) \mbox{ and }r(H) \leq r(G^0(u \ha 2, l))$.
Then
$$l(G) - r(H) \geq l(G^1(s\ha 0, j)) - r(G^0(u \ha 2, l)) = r(J_{s\ha 0}) - \sum_{i=n+1}^j (d_{k_i-1}-d_{k_i}) - l(J_{u\ha 2}) - \sum_{i=n+1}^l (d_{k_i-1}-d_{k_i})$$$$ \geq l(G^0_s) - r(G^1_u) - 2 \sum_{i=n+1}^j(d_{k_i-1}-d_{k_i}) > l(G^0_s) - r(G^1_u) - 2 \sum_{i=n+1}^{\infty}(d_{k_i-1}-d_{k_i}) \geq l(G^0_s) - r(G^1_u)-m_n\geq 0.$$
Consequently, $G \cap H = \emptyset$. 

Now, let $D \in \G^1_s, F \in \G^0_u$. Put $N:=N(1,u)$. Of course, $N < k_n$. Since $d_{k_n} =d_{k_{n-1}}\cdot \frac{1-a_n}{2} < \frac{d_{k_{n-1}}}{2}$, we have $d_{k_n-1} - d_{k_n} > d_{k_n}$.
Thus, using the definition of covering intervals and Lemma \ref{lem} (\ref{1-2}), (\ref{1-2a}), (\ref{1-2b}), we get
$$l(G) \geq l(J_s) = l(J_{u|(N-1)\ha (u_{N}+1)}) \geq l(J_u) + d_{N-1} - d_{N} \geq  l(J_u) + d_{k_n-1}-d_{k_n} > l(J_u) + d_{k_{n-1}}.$$
Similarly, 
$$r(H) \leq r(J_s) - d_{k_{n-1}}.$$
Moreover, if $D$ is a gap of rank $p \geq n$, then
$$l(D) \geq l(G^1(s,p)) \geq r(J_s) - \sum_{i=n}^{\infty} (d_{k_i-1}-d_{k_i})\geq r(J_s) - \sum_{i=k_n}^{\infty} (d_{i-1}-d_{i}) = r(J_s) - d_{k_{n-1}} \geq r(H) \geq r(F)$$
$$r(F) \leq r(G^0(u,p)) \leq l(J_u) + \sum_{i=n}^{\infty} (d_{k_i-1}-d_{k_i}) \leq l(J_u) + \sum_{i=k_n}^{\infty} (d_{i-1}-d_{i}) = l(J_u) + d_{k_{n-1}}\leq l(G)\leq l(D),$$
so $D \cap H = \emptyset$, $G \cap F = \emptyset$ and $D \cap F = \emptyset$.
This finishes the proof of Claim 3.  

We will now prove that \begin{equation} \label{zaw}
J_t \setminus \G_t \subset \bigcap_{n \in \N} \bigcup_{t \prec u, u \in \{0,1,2\}^{k_n}} J_u.
\end{equation}
It suffices to show that for any $n \geq m$ we have
\begin{equation}\label{5*}
J_t \setminus \bigcup_{j=m}^n(\G^0_{t \ha 0^{(k_m-k-1)}}(j) \cup \G^1_{t \ha 2^{(k_m-k-1)}}(j)) \subset \bigcup_{t \prec u, u \in \{0,1,2\}^{k_n}} J_u.
\end{equation}

Let $x \in J_t \setminus(\G^0_{t \ha 0^{(k_m-k-1)}}(m) \cup \G^1_{t \ha 2^{(k_m-k-1)}}(m))$. Since $a_j \leq \frac{1}{3}$ for $j = k+1, \dots, k_m-1$, we have $J_t = \bigcup_{t \prec u, u \in \{0,1,2\}^{k_m-1}} J_u$. Hence $x \in J_s$ for some $ s \in \{0,1,2\}^{k_m-1}, t\prec s$. Further we get $x \in (J_{s\ha 0} \cup J_{s\ha 1}\cup J_{s\ha 2}) \cup (G^0_s \cup G^1_s).$ If $x \in (J_{s\ha 0} \cup J_{s\ha 1}\cup J_{s\ha 2})$, then $x \in \bigcup_{t \prec u, u \in \{0,1,2\}^{k_m}} J_u$. If $ x \in G^i_s$ for $i \in \{0,1\}$ and $G^i_s \notin (\G^0_{t \ha 0^{(k_m-k-1)}}(m) \cup \G^1_{t \ha 2^{(k_m-k-1)}}(m))$, then $N(i,s) \geq k+1$ (Lemma \ref{uw}). Then, by the definition, there exists an interval $J_u$, where $u \in \{0,1,2\}^{k_m-1}, t \prec u$, associated with $J_s$ and covering $\G^i_s$. In particular, $G^i_s \subset_* J_u$. If $i = 0$, then, by Claim 1, we have $G^0_s \subset J_{u \ha 2}$, and if $i = 1$, then $G^1_s \subset J_{u \ha 0}$. Therefore, $x \in \bigcup_{t \prec u, u \in \{0,1,2\}^{k_m}} J_u$.

Let $n \geq m$. Assume that (\ref{5*}) holds for $n$.

We will show that (\ref{5*}) holds for $n+1$.

Let $x \in J_t \setminus \bigcup_{i=m}^{n+1}(\G^0_{t \ha 0^{(k_n-k-1)}}(n+1) \cup \G^1_{t \ha 2^{(k_n-k-1)}}(n+1))$. Then, by the induction hypothesis, we have $x \in \bigcup_{t \prec u, u \in \{0,1,2\}^{k_n}} J_u.$
Since $a_j \leq \frac{1}{3}$ for $j = k_n+1, \dots, k_{n+1}-1$, by Lemma \ref{lem} (\ref{1-6}), we have $x \in \bigcup_{t \prec u, u \in \{0,1,2\}^{k_{n+1}-1}} J_u$. Thus, $x \in J_s$ for some $ s \in \{0,1,2\}^{k_{n+1}-1}, t\prec s$. Further, we have $x \in (J_{s\ha 0} \cup J_{s\ha 1}\cup J_{s\ha 2}) \cup (G^0_s \cup G^1_s).$ If $x \in (J_{s\ha 0} \cup J_{s\ha 1}\cup J_{s\ha 2})$, then $x \in \bigcup_{t \prec u, u \in \{0,1,2\}^{k_{n+1}}} J_u$. If $ x \in G^i_s$ for $i \in \{0,1\}$ and $G^i_s \notin \bigcup_{j=m}^{n+1}(\G^0_{t \ha 0^{(k_m-k-1)}}(j) \cup \G^1_{t \ha 2^{(k_m-k-1)}}(j))$, then from Claim 2 we infer the existence of an interval $J_u$ covering $G^i_s$ such that $t \prec u$, $u \in \{0,1,2\}^{k_j-1}, m \leq j \leq n+1$, and so $J_u$ is an interval covering some family $\G^h_w$, where $h \in \{0,1\}, w \in \{0,1,2\}^{k_j-1}$ and $G^i_s \in \G^h_w$. Hence $x \in J_u$. We have $x \in J_v$ for some $u \prec v$, $v \in \{0,1,2\}^{k_{n+1}}$ or $x \in G^z_q$ for some $z \in \{0,1\}, u \prec q, q \in \{0,1,2\}^{k_l-1}, j \leq l \leq n+1$. In the first case we get the assertion. If $x \in G^z_q$, then $G^z_q \cap G^i_s \neq \emptyset$, so, by Claim 3, we have $G^z_q \notin \G^{1-h}_u$ and $G^z_q \notin \G^{h}_u$. Therefore, by Claim 2, there exists an interval $J_V$, $V \in \{0,1,2\}^{k_f-1}$, $j+1 \leq f \leq n+1$, $u \prec V$ covering $G^z_q$. Since $f > j$, then, repeating the above reasoning finitely many times, we will finally obtain the assertion or that for some gap $G^M_S$, $x \in G^M_S$, where $S \in \{0,1,2\}^{k_{n+1}-1},$ $M \in \{0,1\}$ and there is an interval $J_U$ covering $G^M_S$, where $U \in \{0,1,2\}^{k_{n+1}-1}.$ Then, by Claim 1, $G^M_S \subset J_{U \ha (2(1-M))}$. Thus, $$x \in \bigcup_{t \prec u, u \in \{0,1,2\}^{k_{n+1}}} J_u.$$
By the induction, we obtain (\ref{zaw}).

Of course, $\bigcup_{t \prec u, u \in \{0,1,2\}^{k_n}} J_u \subset C_n(a)-C_n(a)$. Therefore, $J_t \setminus  \G_t = J_t \setminus (\G^0_{t \ha 0^{(k_m-k-1)}} \cup \G^1_{t \ha 2^{(k_m-k-1)}}) \subset C(a)-C(a)$. Moreover, by Theorem \ref{tw4}, $J_t \setminus \G_t$ has nonempty interior. Thus, $C(a)-C(a)$ is a Cantorval.

If $t=\emptyset $,
then, by Theroem \ref{tw4}, we have $C\left( a \right) -C\left( a
\right) \subset \left[ -1,1\right] \setminus \bigcup \G_t$. Hence
$C\left( a \right) -C\left( a \right) =\left[ -1,1\right]
\setminus \bigcup \G_t$. Using Theroem \ref{tw4} again, we obtain
\begin{equation*}
\left\vert C\left( a \right) -C\left( a \right) \right\vert
=2-\left\vert  \bigcup \G_t \right\vert =2-2\sum_{n=1}^{\infty
}3^{n-1}\left( d_{k_{n}-1}-3d_{k_{n}}\right) .
\end{equation*}
\end{proof}

\begin{corollary}\label{wniosek}
Let $a=(a_1,a_2,a_1,a_2,\dots)$, where $a_1<\frac{1}{3},a_2>\frac{1}{3}$. If $a_1\leq\frac{1}{35}$ and 
$$a_2 \leq \frac{-a_1-5+\sqrt{a_1^2+34a_1+33}}{2-2a_1}$$
or $a_1 \in (\frac{1}{35}, \frac{6\sqrt{5}-13}{11})$ and
$$a_2 \leq \frac{3a_1+1-4\sqrt{a_1^2+a_1}}{1-a_1},$$
then the set $C(a)-C(a)$ is a Cantorval.
\end{corollary}
\begin{proof}
Observe that for all $n \in \N$, $k_n = 2n$ and 
$$d_{2n} = \left(\frac{1-a_1}{2}\cdot \frac{1-a_2}{2}\right)^n = d_2^n,$$
$$d_{2n-1} = \left(\frac{1-a_1}{2}\cdot \frac{1-a_2}{2}\right)^{n-1}\cdot\frac{1-a_1}{2} = d_2^{n-1}\cdot d_1,$$
and
$$\delta_n=\Delta_n = 3d_{2n-1}-d_{2n-2}.$$
Hence
$$m_n = \min \{\delta_{n}-(d_{k_{n}-1}-d_{k_{n})}, 4d_{k_{n}}-\Delta_{n} \} = \min\{3d_{2n-1}-d_{2n-2}-(d_{2n-1}-d_{2n}), 4d_{2n}-3d_{2n-1}+d_{2n-2}\}=  $$$$=\min\{d_{2n}+2d_{2n-1}-d_{2n-2}, 4d_{2n}-3d_{2n-1}+d_{2n-2}\}
$$$$=d_2^{n-1} \cdot \min\{d_{2}+2d_{1}-1, 4d_{2}-3d_{1}+1\} = d_2^{n-1} \cdot m_1. $$ Moreover,
$$\sum_{i=n+1}^{\infty}(d_{k_i-1}-d_{k_i})=\sum_{i=n+1}^{\infty}(d_{2i-1}-d_{2i})= d_2^{n}\cdot \sum_{i=1}^{\infty}(d_{2i-1}-d_{2i}).$$
Hence to use Theorem \ref{main} it suffices to show that $m_1 \geq 2d_2\cdot \sum_{i=1}^{\infty}(d_{2i-1}-d_{2i})$. Calculate
$$\sum_{i=1}^{\infty}(d_{2i-1}-d_{2i})=\sum_{i=1}^{\infty}d_2^{i-1}(d_{1}-d_{2})= \frac{d_1-d_2}{1-d_2}.$$
After calculations (using a computer) we obtain that the system of the inequalities
$$ d_{2}+2d_{1}-1 \geq 2d_2\frac{d_1-d_2}{1-d_2}$$
and
$$4d_{2}-3d_{1}+1 \geq 2d_2\frac{d_1-d_2}{1-d_2},$$
which are equivalent to
$$\frac{(1-a_1)(1-a_2)}{4}-a_1 \geq \frac{(1-a_1)^2(1-a_2^2)}{8-2(1-a_1)(1-a_2)} $$
and
$$(1-a_1)(1-a_2)-\frac{3(1-a_1)}{2}+1 \geq \frac{(1-a_1)^2(1-a_2^2)}{8-2(1-a_1)(1-a_2)}, $$
are satisfied if and only if the assumptions of this Corollary are satisfied.
Thus, when a sequence $a$ is as in the formulation of this Corollary, then the assumptions of Theorem \ref{main} are satisfied and $C(a)-C(a)$ is a Cantorval.
\end{proof}

The picture below shows the area described in the above corollary.
\begin{figure} [h]
\includegraphics[width=0.5\textwidth]{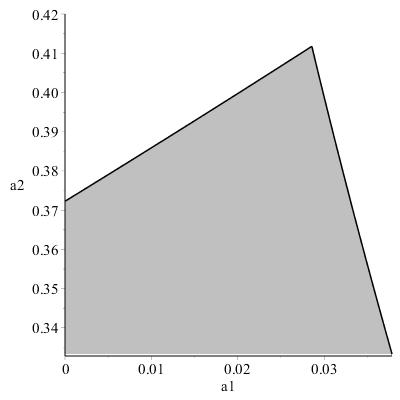}
\caption{}
\label{fig}
\end{figure}

In particular, the point at the top of the above picture is $(\frac{1}{35},\frac{7}{17}).$

In \cite{FN}, there was given another sufficient condition for a set $C(a)-C(a)$ to be a Cantorval.
\begin{theorem} \cite{FN}
\label{Th1}Assume that $a =\left( a _{j}\right) _{j\in \mathbb{N}%
}\in \left( 0,1\right) ^{\mathbb{N}}$ is a sequence such that: $%
a _{n}>\frac{1}{3}$\ for infinitely many terms, $a _{n}\leq \frac{%
1}{3}$\ for infinitely many terms, and $k_{0}\in \mathbb{N}\cup
\left\{ 0\right\} $ is such that $a_{k_{0}+1}<\frac{1}{3}$. Let $k \geq k_0$, $t \in \{0,1,2\}^k$ and $m \in \N$ be such that $k_{m-1} \leq k < k_m$, where the sequence $\left( k_n \right)$ consists of all indices greater than $k_0$, for which $a_{k_n} > \frac{1}{3}$. Denote
$$ \delta _n := \min \{ 3d_i - d_{i-1} \colon i \in \{k_{n-1}+1, \dots, k_n-1\}\},$$
$$ \Delta _n := \max \{ 3d_i - d_{i-1} \colon i \in \{k_{n-1}+1, \dots, k_n-1\}\},$$
where $\max \emptyset =-\infty, \min \emptyset =\infty$.
Put
$$m_n' := \min \{\delta_{n-1}-(d_{k_{n}-1}-d_{k_{n})}, 4d_{k_{n}}-\Delta_{n-1},\delta_n \} $$
$$M_n' := \max \{\delta_{n-1}-(d_{k_{n}-1}-d_{k_{n})}, 4d_{k_{n}}-\Delta_{n-1},\Delta_n \} $$ for $n \in \N$.
If for any $n\in\N$ we have
$$m_n' =M_n' =\sum_{i=n}^{\infty}(d_{k_i-1}-d_{k_i}), $$
then:

\begin{enumerate}
\item \label{Th1-1}$J_{t}\setminus \bigcup \mathcal{G}_{t}\subset C\left(
a \right) -C\left( a \right) $.

\item \label{Th1-2}The set $C\left( a \right) -C\left( a \right) 
$ is a Cantorval.

\item \label{Th1-3}If $k=k_{0}=0$ and $t=\emptyset $, then $C\left( a
\right) -C\left( a \right) =J_{t}\setminus \bigcup \mathcal{G}_{t}$
and 
\begin{equation*}
\left\vert C\left( a \right) -C\left( a \right) \right\vert
=2-2\sum_{n=1}^{\infty }3^{n-1}\left( d_{k_{n}-1}-3d_{k_{n}}\right) .
\end{equation*}
\end{enumerate}
\end{theorem}
\begin{remark}
The assumptions in \cite{FN} were slightly different, but it is not difficult to see that they are equivalent. It follows from the fact that if $m_n' = M_n'$, then also $\delta_n = m_n'$ and the sequence $(\delta_n)$ satisfies the system of equations $(3.1)$ given in \cite{FN}. On the other hand, if there is a sequence satisfying $(3.1)$ from \cite{FN}, then, from the first equality, it must be equal to the sequence $(\delta_n)$ and the other equalities imply that $\delta_n=m_n'=M_n'$. The equality $\delta_n = \sum_{i=n}^{\infty}(d_{k_i-1}-d_{k_i})$ follows from the fact that for all $n$ $\delta_{n+1} = \delta_n -(d_{k_n-1}-d_{k_n})$.
\end{remark}
It is easy to see that the assumptions of Theorem \ref{Th1} and Theorem \ref{main} cannot be satisfied at the same time. Indeed, if assumptions of Theorem \ref{Th1} are satisfied, then, since $M_n' = m_n'$ for all $n$, we have also $m_n'=m_n$. Therefore,
$$m_n=m_n' = \sum_{i=n}^{\infty}(d_{k_i-1}-d_{k_i}) < 2\sum_{i=n}^{\infty}(d_{k_i-1}-d_{k_i}),$$
so the assumption 
$m_n \geq 2\sum_{i=n}^{\infty}(d_{k_i-1}-d_{k_i})$ from Theorem \ref{main} is not satisfied.
Thus, Theorem \ref{main} gives us a new sufficient condition for the set $C(a)-C(a)$ to be a Cantorval.

\section{Final remarks and open problems}
In this section we will provide a short comparison of the obtained results to the known results on achievement sets.
Let us recall the notion of achievement sets. Let $x=\left( x_{j}\right) _{j\in \mathbb{N}}$ be
a nonincreasing sequence of positive numbers such that the series $%
\sum_{j=1}^{\infty }x_{j}$ is convergent. The set%
\begin{equation*}
E\left( x\right) :=\left\{ \sum_{j\in A}x_{j}:A\subset \mathbb{N}\right\} 
\end{equation*}%
(where $\sum_{j \in \emptyset}x_j := 0$) of all subsums of $\sum_{j=1}^{\infty }x_{j}$ is called the
achievement set of $x$. 
The $n$-th remainder of a series is denoted by 
$r_{n}:=\sum_{j=n+1}^{\infty }x_{j}$.
If $x_{n}>r_{n}$ for $n\in \mathbb{N}$, then the series
is called fast convergent.

The known relationship between central Cantor sets and the
achievement sets of fast convergent series is given by the following proposition. 

\begin{proposition}[{\protect\cite[p. 27]{FF}}]
\label{Pr2}The following conditions hold.

\begin{enumerate}
\item If $(a_n) \in (0,1)^\N$ and $\lambda_n = \frac{1-a_n}{2}$ for all $n \in \N$, then the series $\sum_{n=1}^{\infty }x_{n}$ given
by the formula%
\begin{equation}
x_{1}=1-\lambda _{1}\quad \text{and}\quad x_{n}=\lambda _{1}\cdot \ldots
\cdot \lambda _{n-1}\cdot \left( 1-\lambda _{n}\right) \text{ for }n>1,
\label{3-1}
\end{equation}%
is fast convergent, $r_0=1$ and $C\left( a \right) =E\left( x\right) $.

\item If a series $\sum_{n=1}^{\infty }x_{n}$ is fast convergent and $%
\lambda _{n}=\frac{r_{n}}{r_{n-1}}$ for $n\in \mathbb{N}$, then for all $n \in \N$, $a_n:=1-2\lambda_n \in (0,1) $ and $E\left( x\right) =r_0\cdot C\left( a \right) $.
\end{enumerate}
\end{proposition}

There is an important family of sequences, called multigeometric sequences, for which achievement sets are considered. Let $n\in \mathbb{N}$, $x_{1},x_{2},\dots ,x_{n}\in \mathbb{R}$ and $q\in
(0,1)$. The sequence 
\begin{equation*}
\left( x_{1},x_{2},\dots ,x_{n},x_{1}q,x_{2}q,\dots
,x_{n}q,x_{1}q^{2},x_{2}q^{2},\ldots ,x_{n}q^{2},\dots \right)
\end{equation*}%
is called a multigeometric sequence with the ratio $q$ and is denoted
by $\left( x_{1},x_{2},\dots ,x_{n};q\right) $.
  
Observe that since $C(a)$ is symmetric with respect to $\frac{1}{2}$, we have $C(a) = 1-C(a)$, and so $C(a)-C(a)= C(a)+C(a) - 1$, thus the topological structure of $C(a)-C(a)$ and of $C(a)+C(a)$ are the same. 

In \cite{BGM} it is shown that the set $E(3,3,2,2;\frac{1}{9})$ is a Cantorval. It is easy to see that $E(3,3,2,2;q) = E(3,2;q) + E(3,2;q).$ Moreover, if $q < \frac{1}{6}$, then $E(3,2;q)$ is fast convergent, and so $E(3,2;q)$ is a central Cantor set. It is not difficult to calculate, using Proposition \ref{Pr2}, that $E(3,2;q) = r_0 C(a)$, where $a= (a_n)$, $a_{2n-1} = \frac{1-6q}{5}$ and $a_{2n} = \frac{2-7q}{2+3q}$ for $n \in \N$. For $q = \frac{1}{9}$ we have $a_{2n-1} = \frac{1}{15}$ and $a_{2n} = \frac{11}{21}$ for $n \in \N$. Hence for such a sequence we get that $C(a) - C(a)$ is a Cantorval. This example was also considered in \cite{FN} as the conclusion follows also from \cite[Theorem 3.2]{FN}.

From \cite[Theorem 1.3]{BBFS} we can infer that if $q \in [\frac{1}{7}, \frac{1}{6})$, then  $E(3,3,2,2;q)$ is a Cantorval. In particular, $C(a) -C(a)$ is a Cantorval for $a$ defined as above with $q \in [\frac{1}{7}, \frac{1}{6})$. It is worth pointing out that all such sequences $a$ satisfy the assumptions of Theorem \ref{main}, and thus also of Corollary \ref{wniosek}, and so $(a_1,a_2)$ lie in the grey area from Figure \ref{fig}. Moreover, for $q = \frac{1}{7}$, $(a_1,a_2)$ is equal to $(\frac{1}{35},\frac{7}{17})$, that is, to the point at the top of the picture.

It is worth mentioning that from \cite[Theorem 1.3]{BBFS} we can also infer that $E(3,3,2,2;q)$ is a Cantor set for $q< \frac{1}{9}.$ However, it is still an open question what is the topological structure of $E(3,3,2,2;q)$ for $q \in (\frac{1}{9}, \frac{1}{7})$. There are just some partial results from the paper \cite{BBFS}. Namely, it is known that for almost all $q \in (\frac{1}{9}, \frac{1}{7})$, $E(3,3,2,2;q)$ has a positive Lebesgue measure, but also there is a decreasing sequence of $(q_n)$ tending to $\frac{1}{9}$ such that $E(3,3,2,2;q_n)$ is a Cantor set of measure zero. 

To solve the above problem we should find some new conditions which would imply that the set $C(a)-C(a)$ is a Cantor set. However, there is still very little methods to find such a condition, because of a complicated geometry. A great milestone could be made in this matter if the following conjecture would occur to be true. 

\begin{conjecture}
For any $a \in (0,1)^\N$ if $\emph{int}(C(a)-C(a)) \neq \emptyset$, then $0\in \emph{int}(C(a)-C(a))$.
\end{conjecture}

This would be a very useful tool as it would allow to understand better geometry of the sets $C(a)-C(a)$.
\section*{Acknowledgements}
The author would like to express his gratitude to  Tomasz Filipczak for many valuable remarks and suggestions.

This research was funded in whole by National Science Centre, Poland, Grant number: 2022/06/X/ST1/00764.
\section*{Statements and declarations}

The author have no competing interests to declare that are relevant to the content of this article.

\section*{Data Availability Statement}

The paper has no associated data.


\begin{thebibliography}{99}
\bibitem{AC} R. Anisca, C. Chlebovec, \emph{On the structure of arithmetic
sum of Cantor sets with constant ratios of dissection}, Nonlinearity \textbf{%
22} (2009), 2127--2140.

\bibitem{ACI} R. Anisca, C. Chlebovec, M. Ilie, \emph{The structure of arithmetic sums of affine Cantor sets}, Real Anal. Exchange \textbf{%
37(2)} (2011/2012), 324--332.

\bibitem{AI} R. Anisca, M. Ilie, \emph{A technique of studying sums of
central Cantor sets}, Canad. Math. Bull. \textbf{44} (2001), 12--18.

\bibitem{AAV} J. Appell, E. D'Aniello, M. V\"{a}th, \textit{Some remarks on
small sets}, Ric. Math. \textbf{50} (2001), 255--274.

\bibitem{BFN} M. Balcerzak, T. Filipczak, P. Nowakowski, \emph{Families of
symmetric Cantor sets from the category and measure viewpoints}, Georgian
Math. J. \textbf{26} (2019), 545--553.

\bibitem{BBFS} T. Banakh, A. Bartoszewicz, M. Filipczak, E. Szymonik, \emph{%
Topological and measure properties of some self-similar sets}, Topol.
Methods Nonlinear Anal. \textbf{46} (2015), 1013--1028.

\bibitem{BGM} A. Bartoszewicz, S. G{\l }\c{a}b, J. Marchwicki, \emph{%
Recovering a purely atomic finite measure from its range}, J. Math. Anal.
Appl. \textbf{467} (2018), 825--841.


\bibitem{FN} T. Filipczak, P. Nowakowski, \emph{Conditions for the difference set of a central Cantor set to be a Cantorval}, Results Math. \textbf{78} (2023), (art. 166).

\bibitem{GJXZ} J. Gu, K. Jiang, L.Xi, B. Zhao, \emph{Multiplication on uniform $\lambda$-Cantor sets}, Ann. Fenn. Math. \textbf{46} (2021), 703--711.

\bibitem{K1} R. L. Kraft, \emph{Random intersections of thick Cantor sets},
Trans. Amer. Math. Soc. \textbf{352} (2000), 1315--1328.

\bibitem{K} R. L. Kraft, \emph{What's the difference between Cantor sets?},
Amer. Math. Monthly \textbf{101} (1994), 640--650.

\bibitem{MO} P. Mendes, F. Oliveira, \emph{On the topological structure of
the arithmetic sum of two Cantor sets}, Nonlinearity \textbf{7} (1994),
329--343.

\bibitem{Now} P. Nowakowski, \emph{When the algebraic difference of two central Cantor sets is an interval?}, Ann. Fenn. Math. \textbf{48} (2023), 163--185.

\bibitem{FF} F. Prus-Wi\'{s}niowski, F. Tulone, \emph{The arithmetic
decomposition of central Cantor sets}, J. Math. Anal. Appl. \textbf{467}
(2018), 26--31.

\bibitem{S} A. Sannami, \emph{An example of a regular Cantor set whose
difference set is a Cantor set with positive measure}, Hokkaido Math. J. 
\textbf{21} (1992), 7--24.

\bibitem{St} H. Steinhaus, \emph{Nowa własność mnogości Cantora}, Wektor
\textbf{6} (1917), 105--107.

\bibitem{Ta} Y. Takahashi, \emph{Products of two Cantor sets}, Nonlinearity 
\textbf{30} (2017), 2114--2137.

\bibitem{T19} Y. Takahashi, \emph{Sums of two self-similar Cantor sets}, J.
Math. Anal. Appl. \textbf{477} (2019), 613--626.

\bibitem{Tom} B. Thomson, {\it Real functions}, Springer, New York 1985.

\bibitem{Zaj} L. Zaji\v{c}ek, {\it Porosity and $\sigma$-porosity}, Real
Anal. Exchange \textbf{13} (1987), 314--350.
\end{thebibliography}
\end{document}